\documentclass[a4paper,12pt]{article}
\usepackage{geometry}
\usepackage{amsmath,amssymb,amsthm}
\usepackage{amsfonts}
\usepackage{mathptmx} 
\usepackage{amssymb}
\usepackage[mathscr]{euscript}
\usepackage{mathrsfs}
\usepackage{amsthm}\input{amssym.def}
\numberwithin{equation}{section}
\newtheorem{Definition}{Definition}[section]

\newtheorem{Proposition}[Definition]{Proposition}

\newtheorem{Remark}[Definition]{Remark}
\newtheorem{Note}[Definition]{Note}
\title{ \bf \fontsize{21}{25} \selectfont Lowen type multi-fuzzy topological spaces}
\author{\bf \fontsize{14}{17} \selectfont Moumita Chiney and S. K. Samanta\\
\fontsize{12}{15} \selectfont Department of Mathematics, Visva-Bharati,\\ \fontsize{12}{15} \selectfont Santiniketan- 731235, West Bengal, India.\\ \fontsize{12}{15} \selectfont E-mail: moumi.chiney@gmail.com,\;syamal\_123@yahoo.co.in}
\date{}
\begin{document}
\maketitle
\begin{abstract}
In this paper  Lowen type multi-fuzzy topological
space has been introduced and characterization of topology by its nbd system is studied. Also the product multi-fuzzy topological space
has been introduced and it has been investigated that 2nd countability and compactness are finitely productive in multi-fuzzy topological spaces. 
\end{abstract}

\section{Introduction}

Fuzzy set theory, which was first initiated by Zadeh \cite{zadeh}
in 1965, has become a very important tool to solve many complicated
problems arising in the fields of economics, social sciences, engineering,
medical sciences and so forth, involving uncertainties and provides
an appropriate framework for representing vague concepts by allowing
partial membership. Many researchers have worked on theoretical aspects
and applications of fuzzy set theory over the years, such as fuzzy control
systems, fuzzy logic, fuzzy automata, fuzzy topology, fuzzy topological
groups, fuzzy topological vector spaces, fuzzy differentiation etc.
\cite{chang,ferraro,katsaras,lee,lowen,ming,wong}. A new type of
fuzzy set (multi-fuzzy set ) was introduced in a paper of Sebastian
and Ramkrishnan \cite{sebastian} by ordered sequences of membership
function. This notion of multi-fuzzy sets provides a new method to
represents some problems which are difficult to explain in other extensions
of fuzzy set theory\cite{sebastian2}. The topological structure in
this setting has been defined by Sebastian and Ramkrishnan\cite{sebastian1}.
Also, Dey and Pal \cite{Dey} introduced the notion of multi-fuzzy
complex numbers, multi-fuzzy complex sets, multi-fuzzy vector spaces
etc. In this paper  Lowen type multi-fuzzy topology is introduced and characterization of nbd system is done. Continuity of functions is studied. A notion of product of multi-fuzzy topologies is introduced and productive properties of 2nd countability and compactness are also established.

\section{Preliminaries}
\begin{Definition} \cite{sebastian} \label{mftvs1.1} Let $X$ be
a non-empty set and $\{L_{i}:i\in P\}$ be a family of complete lattices.
A multi-fuzzy set $A$ in $X$ is a set of ordered sequences $A=\{\prec x,\mu_{A_{1}}(x),\mu_{A_{2}}(x),...,\mu_{A_{i}}(x),...\succ:x\in X\},$
where $\mu_{A_{i}}(x)\in L_{i}^{X},$ for $i\in P.$ For the sake
of simplicity we denote the multi- fuzzy set $A=\{\prec x,\mu_{A_{1}}(x),\mu_{A_{2}}(x),...,\mu_{A_{i}}(x),....\succ:x\in X\}$
as $A=\prec\mu_{A_{1}},\mu_{A_{2}},...,\mu_{A_{i}},...\succ.$ The
set of all multi-fuzzy sets in $X$ with $\{L_{i}:i\in P\}$ is denoted
by $\underset{i\in P}{\prod}L_{i}^{X}.$ \end{Definition}
\begin{Definition} \cite{sebastian} \label{mftvs1.2} Let $A,B$ be multi-fuzzy sets in $\underset{i\in P}{\prod}L_{i}^{X}.$ Then we have the
following relations and operations.\begin{itemize}
\item[(i)] $A$ is said to be a multi-fuzzy subset of $B,$ denoted by
$A\sqsubseteq B$, if $\mu_{A_{i}}\leq\mu_{B_{i}},$ for all $i\in P;$
\item[(ii)] $A$ is said to be equal to $B,$ denoted by $A=B$, if $\mu_{A_{i}}=\mu_{B_{i}},$
for all $i\in P;$ 
\item[(iii)] Intersection of $A$ and $B$, denoted by $A\sqcap B,$ is
defined by $\mu_{A\sqcap B}=\prec(\mu_{A_{i}}\wedge\mu_{B_{i}})_{i\in P}\succ;$
\item[(iv)] Union of $A$ and $B$, denoted by $A\sqcup B,$ is defined
by $\mu_{A\sqcup B}$ $=\prec(\mu_{A_{i}}\vee\mu_{B_{i}})_{i\in P}\succ;$
\item[(v)] complement of $A$, denoted by $A^{C}$, is defined by $\mu_{A^C}=\prec(\mu_{A^{C}_{i}})_{i\in P}\succ$, where $\mu_{A^{C}_{i}}$ is the complement of the fuzzy set $\mu_{A_{i}}$;
\item[(vi)] a multi-fuzzy set is said to be null multi-fuzzy set, denoted
by $\bar{\Phi},$ if $\mu_{\bar{\Phi}_{i}}=\bar{0},$ for all $i\in P;$
where $\bar{0}$ is the null fuzzy set;
\item[(vii)] a multi-fuzzy set is said to be absolute multi-fuzzy set,
denoted by $\bar{X},$ if $\mu_{\bar{X}_{i}}=\bar{1},$ for all $i\in P;$
where $\bar{1}$ is the absolute fuzzy set.\end{itemize} \end{Definition}
\begin{Definition} \cite{sebastian1} \label{mftvs1.3} A subset $\delta$
of $\underset{i\in P}{\prod}L_{i}^{X}$ is called a multi-fuzzy topology
on $X,$ if it satisfies the following conditions:\begin{itemize}
\item[(i)] $\bar{\Phi},\bar{X}\in\delta;$
\item[(ii)] the intersection of any two multi-fuzzy sets in $\tau$ belongs
to $\tau$. 
\item[(iii)] the union of any number of multi-fuzzy sets in $\tau$ belongs
to $\tau$. \end{itemize}
\noindent The triplet $(X,\underset{i\in P}{\prod}L_{i}^{X},\delta)$
is called multi-fuzzy topological space.The members of $\delta$ are
said to be $\delta-$ open multi-fuzzy sets or simply open multi-fuzzy
sets in $X.$ A multi-fuzzy set $A\in\underset{i\in P}{\prod}L_{i}^{X}$
is called $\delta-$ closed if and only if its complement $A^{C}$
is $\delta-$ open. \end{Definition}
\begin{Remark} \cite{sebastian}\label{mftvs1.4} If the sequences of membership
functions have only $n-$terms (finite number of terms), $n$ is called
the dimension of $A.$ If $L_{i}=[0,1]$ (for $i=1,2,...n$ ), then
the set of all multi-fuzzy sets in $X$ of dimension $n$ is denoted
by $M^{n}FS(X).$ The multi-membership function $\mu_{A}$ is a function
from $X$ to $I^{n}$ such that for all $x$ in $X,$ $\mu_{A}(x)=\prec\mu_{A_{1}}(x),\mu_{A_{2}}(x),...,\mu_{A_{n}}(x)\succ.$
For the sake of simplicity we denote the multi- fuzzy set $A=\{\prec x,\mu_{A_{1}}(x),\mu_{A_{2}}(x),...,\mu_{A_{n}}(x)\succ:x\in X\}$
as $A=\prec\mu_{A_{1}},\mu_{A_{2}},...,\mu_{A_{n}}\succ.$ 

In this paper $I_{i}=[0,1]$ (for $i=1,2,..n$) and $M^{n}FS(X)$ i.e. the set of all multi-fuzzy sets in $X$ of dimension $n$ is denoted by $\overset{n}{\underset{i=1}{\prod}}I_{i}^{X}$. $\newline$\\
Following Sebastian et al. \cite{sebastian1} some definitions and
preliminary results are presented in the rest part of this section
in our form.\end{Remark}

\begin{Definition} \label{mftvs1.6}Let $X$ and $Y$ be two non-empty sets
and $f:X\rightarrow Y$ be a mapping. Then\begin{itemize}
\item[(i)] the image of a multi-fuzzy set $A$ $\in$ $\overset{n}{\underset{i=1}{\prod}}I_{i}^{X}$
under the mapping $f$ is a multi-fuzzy set $\overset{n}{\underset{i=1}{\prod}}I_{i}^{Y}$,
which is defined by $\mu_{f(A)}(y)=\underset{y=f(x)}{\vee}\mu_{A}(x),$
$A\in\overset{n}{\underset{i=1}{\prod}}I_{i}^{X},y\in Y.$ 

\item[(ii)] the inverse image of a multi-fuzzy set $B$ $\in$
$\overset{n}{\underset{i=1}{\prod}}I_{i}^{Y}$ under the mapping $f$ is a multi
fuzzy set in $\overset{n}{\underset{i=1}{\prod}}I_{i}^{X}$, which is defined
by $\mu_{f^{-1}(B)}(x)=\mu_{B}(f(x)),$ $B\in\overset{n}{\underset{i=1}{\prod}}I_{i}^{Y},x\in X.$
\end{itemize}
\end{Definition}
\begin{Proposition} \label{mftvs1.7} Let $f:X\rightarrow Y$ be a mapping
and $F^{1},F^{2},F^{k}\in\overset{n}{\underset{i=1}{\prod}}I_{i}^{X},$ $G,G^{1},G^{2},G^{k}\in\overset{n}{\underset{i=1}{\prod}}I_{i}^{Y}$,
for $k\in\triangle.$ Then \begin{itemize}
\item[(i)] $f(\bar{\Phi}_{X})=\bar{\Phi}_{Y};$ 
\item[(ii)] $F^{1}\sqsubseteq F^{2}$ implies $f(F^{1})\sqsubseteq f(F^{2});$
\item[(iii)] $f(\underset{k\in\triangle}{\sqcap}F^{k})\sqsubseteq\underset{k\in\triangle}{\sqcap}f(F^{k});$
\item[(iv)] $f(\underset{k\in\triangle}{\sqcup}F^{k})=\underset{k\in\triangle}{\sqcup}f(F^{k});$
\item[(v)] $f^{-1}(\bar{\Phi}_{Y})=\bar{\Phi}_{X}$ and $f^{-1}(\bar{Y})=\bar{X};$
\item[(vi)] $G^{1}\sqsubseteq G^{2}$ implies $f^{-1}(G^{1})\sqsubseteq f^{-1}(G^{2});$
\item[(vii)] $f^{-1}(\underset{k\in\triangle}{\sqcup}G^{k})=\underset{k\in\triangle}{\sqcup}f^{-1}(G^{k});$
\item[(viii)] $f^{-1}(\underset{k\in\triangle}{\sqcap}G^{k})=\underset{k\in\triangle}{\sqcap}f^{-1}(G^{k});$
\item[(ix)]  $f^{-1}(G^{C})=[f^{-1}(G)]^{C};$
\item[(x)] $F^{k}\sqsubseteq f^{-1}(f(F^{k})),$ equality holds if $f$
is injective;
\item[(xi)] $f(f^{-1}(G^{k}))\sqsubseteq G^{k},$ equality holds if $f$
is surjective.\end{itemize}\end{Proposition}

\section{Lowen type multi-fuzzy topology}
Unless otherwise mentioned by $\hat{0}$ we denote the n-tuple $(0,0,..,0)$. 
\begin{Definition} \label{mftvs1.5} A multi-fuzzy set $A$ in $\overset{n}{\underset{i=1}{\prod}}I_{i}^{X}$ is said to be non-null constant multi-fuzzy set if $\mu_{A_{i}}=\bar{c_{i}},$
for all $i=1,2,..,n;$ where $\bar{c_{i}}(x)=c_{i},$ $\forall x\in X$
is a constant fuzzy set with $c_{i}\in(0,1].$ This is denoted by $C_X^n$.\end{Definition}
\begin{Note} The null multi-fuzzy set in $\overset{n}{\underset{i=1}{\prod}}I_{i}^{X}$ as defined in Definition \ref{mftvs1.2} $(vi)$ will be denoted by $\Phi_{X}^{n}$.
\end{Note}
\begin{Definition} \label{mftvs2.7} For a multi-fuzzy set $F\in\overset{n}{\underset{i=1}{\prod}}I_{i}^{X},$ \begin{itemize}
\item[(1)] $\mu_{F}(x)\succ\hat{0}\Leftrightarrow\mu_{F_{i}}(x)>0,$ for all
$i=1,2,..,n.$
\item[(2)] $\mu_{F}(x)= \hat{0}\Leftrightarrow\mu_{F_{i}}(x)=0,$ for all
$i=1,2,..,n.$
\end{itemize}\end{Definition}
\begin{Definition} Let $\mathcal{M}_X$ denote the collection of all multi-fuzzy sets in $\overset{n}{\underset{i=1}{\prod}}I_{i}^{X}$ such that $F \in \mathcal{M}_{X}$ $\implies$ for any $x\in X,$ either $\mu_{F}(x)\succ\hat{0}$ or $\mu_{F}(x)= \hat{0}.$ 
\end{Definition}
\begin{Definition} \label{mftvs2.1} Let $\tau$ be a sub-collection
of $\mathcal{M}_{X}.$ Then $\tau$ is said to be
a Lowen type multi-fuzzy topology on $X$ if\begin{itemize}
\item[(i)]  $\Phi_{X}^{n},~C_{X}^{n}\in\tau$, where $\mu_{C_{X,i}^{n}}=\bar{c_{i}}$, for
all $i=1,2,..,n,$ and $\bar{c_{i}}(x)=c_{i},$ $\forall x\in X$
is constant fuzzy set with $c_{i}\in(0,1]$.
\item[(ii)] the intersection of any two multi-fuzzy sets in $\tau$ belongs
to $\tau$. 
\item[(iii)] the union of any number of multi-fuzzy sets in $\tau$ belongs
to $\tau$. \end{itemize}
\noindent The triplet $(X,\overset{n}{\underset{i=1}{\prod}}I_{i}^{X},\tau)$
is called Lowen type multi-fuzzy topological space over $X$. \\
The members of $\tau$  are called multi-fuzzy open sets and a multi-fuzzy set is said to be multi-fuzzy closed if its complement is in $\tau$.  \end{Definition}

\begin{Definition} \label{mftvs2.2} Let $(X,\overset{n}{\underset{i=1}{\prod}}I_{i}^{X},\tau)$
be a Lowen type multi-fuzzy topological space. A sub-collection $\mathcal{B}$
of $\tau$ is said to be an open base of $\tau$ if every member of
$\tau$ can be expressed as the union of some members of $\mathcal{B}$.\end{Definition}

\begin{Definition}\label{mftvs2.3} Let $(X,\overset{n}{\underset{i=1}{\prod}}I_{i}^{X},\tau)$
and $(Y,\overset{n}{\underset{i=1}{\prod}}I_{i}^{Y},\nu)$ be two Lowen type
multi-fuzzy topological spaces. A mapping $f:(X,\overset{n}{\underset{i=1}{\prod}}I_{i}^{X},\tau)\rightarrow(Y,\overset{n}{\underset{i=1}{\prod}}I_{i}^{Y},\nu)$
is said to be \begin{itemize}
\item[(i)] multi-fuzzy continuous if $f^{-1}(A)\in\tau,\forall A\in\nu$. 
\item[(ii)] multi-fuzzy homeomorphism if $f$ is bijective and
$f,f^{-1}$ are multi-fuzzy continuous. 
\item[(iii)] multi-fuzzy open if $A\in\tau$ $\Rightarrow$
$f(A)\in\nu$; 
\item[(iv)] multi-fuzzy closed if $A$ is multi-fuzzy closed
in $(X,\overset{n}{\underset{i=1}{\prod}}I_{i}^{X},\tau)$ $\Rightarrow$ $f(A)$
is multi-fuzzy closed in $(Y,\overset{n}{\underset{i=1}{\prod}}I_{i}^{Y},\nu)$.\end{itemize}
\end{Definition}

\begin{Proposition} \label{mftvs2.4}Let $(X,\overset{n}{\underset{i=1}{\prod}}I_{i}^{X},\tau)$,
$(Y,\overset{n}{\underset{i=1}{\prod}}I_{i}^{Y},\nu)$ and $(Z,\overset{n}{\underset{i=1}{\prod}}I_{i}^{Z},\omega)$
be Lowen type multi-fuzzy topological spaces. If $f:(X,\overset{n}{\underset{i=1}{\prod}}I_{i}^{X},\tau)\rightarrow(Y,\overset{n}{\underset{i=1}{\prod}}I_{i}^{Y},\nu)$
and $g:(Y,\overset{n}{\underset{i=1}{\prod}}I_{i}^{Y},\nu)\rightarrow(Z,\overset{n}{\underset{i=1}{\prod}}I_{i}^{Z},\omega)$
are multi-fuzzy continuous (open) and $f(X)\subseteq Y$, then the
composition $g\circ f:(X,\overset{n}{\underset{i=1}{\prod}}I_{i}^{X},\tau)\rightarrow(Z,\overset{n}{\underset{i=1}{\prod}}I_{i}^{Z},\omega)$
is multi-fuzzy continuous (open).\end{Proposition}
\begin{proof} Let $H\in\omega.$ Since $g:(Y,\overset{n}{\underset{i=1}{\prod}}I_{i}^{Y},\nu)\rightarrow(Z,\overset{n}{\underset{i=1}{\prod}}I_{i}^{Z},\omega)$
is a multi-fuzzy continuous mapping, it follows that $g^{-1}(H)\in\nu.$
Again since $f:(X,\overset{n}{\underset{i=1}{\prod}}I_{i}^{X},\tau)\rightarrow(Y,\overset{n}{\underset{i=1}{\prod}}I_{i}^{Y},\nu)$
is a multi-fuzzy continuous mapping, it follows that $f^{-1}(g^{-1}(H))\in\tau\Rightarrow(f\circ g)^{-1}(H)\in\tau.$
So, the composition $g\circ f:(X,\overset{n}{\underset{i=1}{\prod}}I_{i}^{X},\tau)\rightarrow(Z,\overset{n}{\underset{i=1}{\prod}}I_{i}^{Z},\omega)$
is multi-fuzzy continuous.\\
The proof in the case of multi-fuzzy open mapping is similar.\end{proof}

\begin{Proposition} \label{mftvs2.5} Let $(X,\overset{n}{\underset{i=1}{\prod}}I_{i}^{X},\tau)$
and $(Y,\overset{n}{\underset{i=1}{\prod}}I_{i}^{Y},\nu)$ be two Lowen type
multi-fuzzy topological spaces. For a bijective mapping $f:(X,\overset{n}{\underset{i=1}{\prod}}I_{i}^{X},\tau)\rightarrow(Y,\overset{n}{\underset{i=1}{\prod}}I_{i}^{Y},\nu)$,
the following statements are equivalent: \begin{itemize}
\item[(i)] $f:(X,\overset{n}{\underset{i=1}{\prod}}I_{i}^{X},\tau)\rightarrow(Y,\overset{n}{\underset{i=1}{\prod}}I_{i}^{Y},\nu)$
is multi-fuzzy homeomorphism; 
\item[(ii)] $f:(X,\overset{n}{\underset{i=1}{\prod}}I_{i}^{X},\tau)\rightarrow(Y,\overset{n}{\underset{i=1}{\prod}}I_{i}^{Y},\nu)$
and $f^{-1}:(Y,\overset{n}{\underset{i=1}{\prod}}I_{i}^{Y},\nu)\rightarrow(X,\overset{n}{\underset{i=1}{\prod}}I_{i}^{X},\tau)$
are multi-fuzzy continuous; 
\item[(iii)] $f:(X,\overset{n}{\underset{i=1}{\prod}}I_{i}^{X},\tau)\rightarrow(Y,\overset{n}{\underset{i=1}{\prod}}I_{i}^{Y},\nu)$
is both multi-fuzzy continuous and multi-fuzzy open. 
\end{itemize}
\end{Proposition}
\begin{Definition}\label{mftvs2.6} If $\tau_{j},j\in J$ are Lowen type multi
fuzzy topologies in $\overset{n}{\underset{i=1}{\prod}}I_{i}^{X},$ then $\underset{j\in J}{\cap}\tau_{j}$
is a multi-fuzzy topology in $\overset{n}{\underset{i=1}{\prod}}I_{i}^{X}.$\end{Definition}
\begin{Definition} \label{mftvs2.8} A multi-fuzzy set $F$ in a Lowen type
multi-fuzzy topological space $(X,\overset{n}{\underset{i=1}{\prod}}I_{i}^{X},\tau)$
is said to be a multi-fuzzy neighbourhood of a point $x\in X$ if there is a $G\in\tau$
such that $G\sqsubseteq F$ and $\mu_{F}(x)= \mu_{G}(x)\succ \hat{0}$.\\
By $\mathcal{F}_{x}$ we denote the family of all multi-fuzzy neighbourhoods of
$x$ which are determined by the multi-fuzzy topology $\tau$ on $X.$
\end{Definition}
\begin{Proposition}\label{mftvs2.9} If $F_{x}$ and $G_{x}$ are multi-fuzzy neighbourhoods
of $x,$ then $F_{x}\sqcap G_{x}$ is also a multi-fuzzy neighbourhood of $x.$ \end{Proposition}
\begin{Proposition}\label{mftvs2.12} A multi-fuzzy set $A$ in $\overset{n}{\underset{i=1}{\prod}}I_{i}^{X}$ is open in the Lowen type multi-fuzzy topological space $(X,\overset{n}{\underset{i=1}{\prod}}I_{i}^{X},\tau)$
iff for every $x\in X$ satisfying $\mu_{A}(x)\succ\hat{0},$ there
is $F_{x}\sqsubseteq A,F_{x}\in\tau$ and $\mu_{F_{x}}(x)=\mu_{A}(x).$
\end{Proposition}
\begin{Proposition}\label{mftvs2.14} Let $(X,\overset{n}{\underset{i=1}{\prod}}I_{i}^{X},\tau)$
be a Lowen type multi-fuzzy topological space. Then for each $x\in X,$
the family $\mathcal{F}_{x}$ of all multi-fuzzy neighbourhoods of $x$ satisfies:\\
$(i)$ every non-null constant multi-fuzzy set belongs to $\mathcal{F}_{x}.$\\
$(ii)$ $N(x)\succ\hat{0},$ for each $N\in\mathcal{F}_{x}.$ \\
$(iii)$ If $N,M\in\mathcal{F}_{x},$ then $N\sqcap M\in\mathcal{F}_{x}.$
\\
$(iv)$ Let $F\in\overset{n}{\underset{i=1}{\prod}}I_{i}^{X}$ and $x\in X$
with $\mu_{F}(x)\succ\hat{0}$.
If for each $i\in \{1,2,...,n\}$ and for each $0<r_{i}<\mu_{F_{i}}(x)$
there exists $F^{r_{i}}\in\mathcal{F}_{x}$ with $F^{r_{i}}\sqsubseteq F$
and $\mu_{F_{i}^{r_{i}}}(x)>r_{i},$ then $F\in\mathcal{F}_{x}.$\\
$(v)$ If $N\in\mathcal{F}_{x},$ then there exists $G\in\mathcal{F}_{x}$
such that $G\sqsubseteq N$ and $\mu_{N}(x)=\mu_{W}(x)$ and if $\mu_{G}(y)\succ\hat{0},$
then $G\in\mathcal{F}_{y}.$\end{Proposition}
\begin{proof} We only give the proof of $(iv).$ We see that if $N\in\mathcal{F}_{x},$
$N\sqsubseteq W$ and $\mu_{N}(x)=\mu_{W}(x),$ then $W\in\mathcal{F}_{x}$
$.......(*)$ and if $N_{j}\in\mathcal{F}_{x},$ $j\in\triangle,$ then $\underset{j\in\triangle}{\sqcup}N_{i}\in\mathcal{F}_{x}$$.........(**)$.
Let $F\in\overset{n}{\underset{i=1}{\prod}}I_{i}^{X}$ and $x\in X$ with $\mu_{F}(x)\succ\hat{0}$. Also let for each $i\in \{1,2,..,n\}$ and for each $0<r_{i}<\mu_{F_{i}}(x)$ there
exists $F^{r_{i}}\in\mathcal{F}_{x}$ with $F^{r_{i}}\sqsubseteq F$
and $\mu_{F_{i}^{r_{i}}}(x)>r_{i}$. Choose $j\in \{1,2,..,n\}.$ Then
for each $0<r_{j}<\mu_{F_{j}}(x)$ there exists $F^{r_{j}}\in\mathcal{F}_{x}$
with $F^{r_{j}}\sqsubseteq F$ and $\mu_{F_{j}^{r_{j}}}(x)>r_{j}$.
Let $\mu_{F_{i}^{j}}(x)=\underset{0<r_{j}<\mu_{F_{j}}(x)}{\vee}\mu_{F_{i}^{r_{j}}}(x),$
for all $i=1,2,..,n.$ Then $F^{j}\sqsubseteq F$ and $\mu_{F_{j}^{j}}(x)=\mu_{F_{j}}(x).$
Thus for each $i\in \{1,2,...,n\},$ there exists $F^{i}$ such that $F^{i}\sqsubseteq F$
and $\mu_{F_{i}^{i}}(x)=\mu_{F_{i}}(x).$ Let $F_{0}=\sqcup\{F^{i}:i=1,2,...,n\}$. Then $\mu_{F_{0}}(x)=\mu_{F}(x)$ and $F_{0}\sqsubseteq F.$ Hence
by $(*)$ and $(**),$ $F\in\mathcal{F}_{x}.$\end{proof}
\begin{Definition}\label{mftvs2.15} Let $X$ be a set and $\mathcal{L}$
be a function from $X$ into the power set of $\overset{n}{\underset{i=1}{\prod}}I_{i}^{X}$.
Then $\mathcal{L}$ is called a multi-fuzzy neighbourhood system on
$X$ if $\mathcal{L}$ satisfies:\\
$(N1)$ for each $x\in X,$ every non-null constant multi-fuzzy set
belongs to $\mathcal{L}(x).$\\
$(N2)$ If $N\in\mathcal{L}(x),$ then $\mu_{N}(x)\succ\hat{0}.$\\
$(N3)$ If $N,M\in\mathcal{L}(x),$ then $N\sqcap M\in\mathcal{L}(x).$
\\
$(N4)$ Let $F\in\overset{n}{\underset{i=1}{\prod}}I_{i}^{X}$ and $x\in X$
with $\mu_{F}(x)\succ\hat{0}$. If for each $i\in\{1,2,..,n\}$ and for each $0<r_{i}<\mu_{F_{i}}(x)$ there exists $F^{r_{i}}\in\mathcal{L}(x)$ with $F^{r_{i}}\sqsubseteq F$
and $\mu_{F_{i}^{r_{i}}}(x)>r_{i},$ then $F\in\mathcal{L}(x).$\\
$(N5)$ If $N\in\mathcal{L}(x),$ then there exists $G\in\mathcal{L}(x)$
such that $G\sqsubseteq N$ and $\mu_{N}(x)=\mu_{G}(x)$ and if $\mu_{G}(y)\succ\hat{0},$
then $G\in\mathcal{L}(y).$
\end{Definition}

\begin{Proposition}\label{mftvs2.16} If $\mathcal{L}$ is a multi-fuzzy neighbourhood system on $X,$ 
we define $\tau_{\mathcal{L}}$ as the family of all multi-fuzzy sets
$G$ in $X$ with the property that if $\mu_{G}(x)\succ\hat{0},$
then $G\in\mathcal{L}(x)$ and $\varPhi\in\tau_{\mathcal{L}}.$ Then $(X,\overset{n}{\underset{i=1}{\prod}}I_{i}^{X},\tau_{\mathcal{L}})$
is a Lowen type multi-fuzzy topological space. Also, for every $x\in X,$ the family $\mathcal{F}_{x}$
of all multi-fuzzy neighbourhoods of $x$ with respect to the multi-fuzzy topology
$\tau_{\mathcal{L}}$ is exactly $\mathcal{L}(x).$\end{Proposition}
\begin{proof} Clearly every constant multi-fuzzy set belongs to $\tau_{\mathcal{L}}.$
Now we see that from $(N4),$ we have $N\in\mathcal{L}(x)$, $N\sqsubseteq W$
and $\mu_{N}(x)=\mu_{W}(x),$ then $W\in\mathcal{L}(x).$$......(*)$\\
and if $N_{j}\in\mathcal{L}(x),$ $j\in\triangle,$ then $\underset{j\in\triangle}{\sqcup}N_{i}\in\mathcal{L}(x).$$.......(**)$
\\
For each $j\in\triangle,$ let $G_{j}\in\tau_{\mathcal{L}}.$ Set
$H=\underset{j\in\triangle}{\sqcup}G_{j}$. If $\mu_{H}(x)\succ\hat{0},$
then there exists nonempty $J_{x}\subset\triangle$ such that $\mu_{G_{j}}(x)\succ\hat{0}$
for all $j\in J_{x}$ and $\underset{j\in\triangle}{\vee}\mu_{G_{j}}(x)=\underset{j\in J_{x}}{\vee}\mu_{G_{j}}(x).$
By definition of $\tau_{\mathcal{L}},$ if $j\in J_{x},$ $G_{j}\in\mathcal{L}(x).$
From $(**)$ we conclude that $\underset{j\in J_{x}}{\sqcup}G_{j}\in\mathcal{L}(x).$
By $(*)$, $H\in\mathcal{L}(x).$ Therefore $H\in\tau_{\mathcal{L}}.$
\\
Let $G,H\in\tau_{\mathcal{L}}$ and suppose that $\mu_{G\sqcap H}(x)\succ \hat{0}.$
Then $\mu_{G}(x)\succ\hat{0}$ and $\mu_{H}(x)\succ\hat{0}$. Hence,
by the definition of $\tau_{\mathcal{L}}$, $G$ and $H$ are in $\mathcal{L}(x).$
It follows from $(N3)$ that $G\sqcap H\in\mathcal{L}(x).$ Thus $G\sqcap H\in\tau_{\mathcal{L}}.$\\
If $N\in\mathcal{F}_{x},$ then there exists $G\in\tau_{\mathcal{L}}$
such that $G\sqsubseteq N$ and $\mu_{N}(x)=\mu_{G}(x)\succ\hat{0}.$
By definition of $\tau_{\mathcal{L}}$, $G\in\mathcal{L}(x).$ \\
Conversely let $N\in\mathcal{L}(x).$ Then by $(N2),$ $\mu_{N}(x)\succ\hat{0}$
and by $(N5)$ there exists $G\in\mathcal{L}(x)$ such that $G\sqsubseteq N$
and $\mu_{G}(x)=\mu_{N}(x)$ and if $\mu_{G}(y)\succ\hat{0},$ then
$G\in\mathcal{L}(y).$ It follows from definition of $\tau_{\mathcal{L}}$,
$G\in\tau_{\mathcal{L}}.$ Consequently $N\in\mathcal{F}_{x}.$ \end{proof}

\begin{Proposition}\label{mftvs2.18} Let $(X,\overset{n}{\underset{i=1}{\prod}}I_{i}^{X},\tau)$
be a Lowen type multi-fuzzy topological space and $\mathcal{L}_{\tau}$
be a function from $X$ to power set of $\overset{n}{\underset{i=1}{\prod}}I_{i}^{X},$
defined by $\mathcal{L}_{\tau}(x)=\mathcal{F}_{x},$ where $x\in X$
and $\mathcal{F}_{x}$ is the family of all multi-fuzzy neighbourhoods of $x$
with respect to $\tau.$ Then $\mathcal{L}_{\tau}$ is a multi-fuzzy
neighbourhood system on $X$ and $\tau_{\mathcal{L}_{\tau}}=\tau.$ 
\end{Proposition}
\begin{proof} By Proposition \ref{mftvs2.14}, $\mathcal{L}_{\tau}$ satisfies the
conditions $(N1)$ to $(N5)$ and therefore is a multi-fuzzy neighbourhood
system on $X.$ By Proposition \ref{mftvs2.16}, $\tau_{\mathcal{L}_{\tau}}$
is a multi-fuzzy topology on $X.$ Also, from Proposition \ref{mftvs2.16}, we can
say that for each $x\in X,$ the multi-fuzzy neighbourhoods of $x$ with respect
to $\tau_{\mathcal{L}_{\tau}}$ are exactly same as the members of
$\mathcal{F}_{x}.$ Since a multi-fuzzy set $U$ is open iff it is a multi-fuzzy neighbourhood
of each point $x$ satisfying $\mu_{U}(x)\succ \hat{0}$, it follows that
$\tau_{\mathcal{L}_{\tau}}=\tau.$\end{proof}
\begin{Proposition}\label{mftvs2.19} Let $(X,\overset{n}{\underset{i=1}{\prod}}I_{i}^{X},\tau)$
and $(Y,\overset{n}{\underset{i=1}{\prod}}I_{i}^{Y},\nu)$ be two multi-fuzzy
topological spaces and $f:X\rightarrow Y$ be any map. Then the following
conditions are equivalent:\\
$(a)$ The function $f$ is multi-fuzzy continuous. \\
$(b)$ The inverse image of every multi-fuzzy closed set is multi
fuzzy closed.\\
$(c)$ For every $x\in X$ and every multi-fuzzy nbd $N$ of $f(x),$
$f^{-1}(N)$ is a multi-fuzzy nbd of $x.$ \\
$(d)$ For every $x\in X$ and every multi-fuzzy nbd $N$ of $f(x)$,
there is a multi-fuzzy neighbourhood $M$ of $x$ such that $f(M)\sqsubseteq N$
and $\mu_{M}(x)=\mu_{f^{-1}(N)}(x).$ \end{Proposition}

\section{Product multi-fuzzy topology}
Unless otherwise mentioned, in the rest part of this paper, multi-fuzzy topological space means Lowen type multi-fuzzy topological space.
\begin{Definition}\label{mftvs3.1} Let $F,G$ be two fuzzy subsets
of $X$ and $Y$ respectively. Then their product, denoted by $F\times G,$
is defined by $\mu_{\left(F\times G\right)}(x,y)=min\{\mu_{F}(x),$
$\mu_{G}(y)\},\forall(x,y)\in X\times Y.$ \end{Definition}

\begin{Definition}\label{mftvs3.2} Let $F\in\overset{n}{\underset{i=1}{\prod}}I_{i}^{X}$
and $G\in\overset{n}{\underset{i=1}{\prod}}I_{i}^{Y}$ respectively. Then their
product is defined by $\mu_{(F\times G)_{i}}=\mu_{F_{i}\times G_{i}},$
for $i=1,2,...,n$, i.e. $F\times G=\prec\left(\mu_{F_{i}}(x)\wedge\mu_{G_{i}}(y)\right)_{i=1}^{n}\succ.$
It is clear that $F\times G$ is a multi-fuzzy set over $X\times Y$,
i.e. $F\times G\in\overset{n}{\underset{i=1}{\prod}}I_{i}^{X\times Y}.$ \end{Definition}
\begin{Definition}\label{mftvs3.3} Let $F\in I^{X}$ and $G\in\overset{n}{\underset{i=1}{\prod}}I_{i}^{Y}$.
Then their product is defined by $\mu_{(F\times G)_{i}}=\mu_{F\times G_{i}},$
for $i=1,2,...,n$, i.e. $F\times G=\prec\left(\mu_{F}(x)\wedge\mu_{G_{i}}(y)\right)_{i=1}^{n}\succ.$
Also, $F\times G$ is a multi-fuzzy set over $X\times Y$, i.e. $F\times G\in\overset{n}{\underset{i=1}{\prod}}I_{i}^{X\times Y}.$
\end{Definition}

\begin{Proposition}\label{mftvs3.4} Let $(X,\overset{n}{\underset{i=1}{\prod}}I_{i}^{X},\tau)$
and $(Y,\overset{n}{\underset{i=1}{\prod}}I_{i}^{Y},\nu)$ be two multi-fuzzy topological spaces. Then $\mathcal{F}=\{F\times G:F\in\tau,G\in\nu\}$
forms an open base for a multi-fuzzy topology on $X\times Y$. \end{Proposition}

\begin{Definition} \label{mftvs3.5}Let $(X,\overset{n}{\underset{i=1}{\prod}}I_{i}^{X},\tau)$
and $(Y,\overset{n}{\underset{i=1}{\prod}}I_{i}^{Y},\nu)$ be two multi-fuzzy topological spaces. The multi-fuzzy topology in $X\times Y$
induced by the open base $\mathcal{F}=\{F\times G:A\in\tau,G\in\nu\}$
is said to be the product multi-fuzzy topology of the multi-fuzzy
topologies $\tau$ and $\nu$. It is denoted by $\tau\times\nu$.
The multi-fuzzy topological space $[X\times Y,\overset{n}{\underset{i=1}{\prod}}I_{i}^{X\times Y},\tau\times\nu]$
is said to be the multi-fuzzy topological product of the multi-fuzzy
topological spaces $(X,\overset{n}{\underset{i=1}{\prod}}I_{i}^{X},\tau)$ and
$(Y,\overset{n}{\underset{i=1}{\prod}}I_{i}^{Y},\nu)$. \end{Definition} 

\begin{Proposition}\label{mftvs3.6} Let $(X,\overset{n}{\underset{i=1}{\prod}}I_{i}^{X},\tau)$
be the product space of two multi-fuzzy topological spaces $(X_{1},\overset{n}{\underset{i=1}{\prod}}I_{i}^{X_{1}},\tau_{1})$
and $(X_{2},\overset{n}{\underset{i=1}{\prod}}I_{i}^{X_{2}},\tau_{2})$ respectively.
Then the projection mappings $\pi_{j}:(X,\overset{n}{\underset{i=1}{\prod}}I_{i}^{X},\tau)\rightarrow(X_{j},\overset{n}{\underset{i=1}{\prod}}I_{i}^{X_{j}},\tau_{j}),j=1,2$
are multi-fuzzy continuous and multi-fuzzy open. Also $\tau_{1}\times\tau_{2}$
is the smallest multi-fuzzy topology in $X_{1}\times X_{2}$ for which
the projection mappings are multi-fuzzy continuous. \\
If further, $(Y,\overset{n}{\underset{i=1}{\prod}}I_{i}^{X},\nu)$ be any multi
fuzzy topological space then the mapping $f:(Y,\overset{n}{\underset{i=1}{\prod}}I_{i}^{Y},\nu)\rightarrow(X,\overset{n}{\underset{i=1}{\prod}}I_{i}^{X},\tau)$
is multi-fuzzy continuous iff the mappings $\pi_{j}\circ f:(Y,\overset{n}{\underset{i=1}{\prod}}I_{i}^{Y},\nu)\rightarrow(X_{j},\overset{n}{\underset{i=1}{\prod}}I_{i}^{X},\tau_{j}),j=1,2$
are multi-fuzzy continuous. \end{Proposition}

\begin{Proposition} \label{mftvs3.7}Let $(X,\overset{n}{\underset{i=1}{\prod}}I_{i}^{X},\tau)$
be a multi-fuzzy topological space. Then the mapping $f:(X,\overset{n}{\underset{i=1}{\prod}}I_{i}^{X},\tau)\rightarrow(X,\overset{n}{\underset{i=1}{\prod}}I_{i}^{X},\tau)$
defined by $f(x)=x,$ $\forall x\in X$ is multi-fuzzy continuous.\end{Proposition}
\begin{proof} The proof is straightforward.
\end{proof}

\begin{Proposition}\label{mftvs3.8} Let $(X,\overset{n}{\underset{i=1}{\prod}}I_{i}^{X},\tau)$
and $(Y,\overset{n}{\underset{i=1}{\prod}}I_{i}^{Y},\nu)$ be two multi-fuzzy
topological spaces. Then the mapping $f:(X,\overset{n}{\underset{i=1}{\prod}}I_{i}^{X},\tau)\rightarrow(Y,\overset{n}{\underset{i=1}{\prod}}I_{i}^{Y},\nu)$
defined by $f(x)=y_{0,}$ $\forall x\in X,$ where $y_{0}$ is a fixed
element of $Y$ is multi-fuzzy continuous.\end{Proposition}
\begin{proof} Let $F\in\nu.$ Then for each $i=1,2,..,n,$ $\mu_{\left(f^{-1}(F)\right)_{i}}(x)=\mu_{f^{-1}(F_{i})}(x)=\mu_{F_{i}}(f(x))=\mu_{F_{i}}(y_{0})=c_{i}$(say),
$\forall x\in X.$ So, $\left(f^{-1}(F)\right)_{i}=\overline{c_{i}}$. If $c_{i}\neq 0,$
let $C_{X}^{n}$ be the multi-fuzzy set such that $\mu_{({C_{X}^{n}})_{i}}=\overline{c_{i}},$
for $i=1,2,..,n.$ Then $f^{-1}(F)=C_{X}^{n}\in\tau.$ Also, if $c_{i}=0,$ then $f^{-1}(F)=\Phi_{X}^{n}\in \tau$. Therefore $f$ is multi-fuzzy continuous. \end{proof}

\begin{Proposition}\label{mftvs3.9} Let $(X,\overset{n}{\underset{i=1}{\prod}}I_{i}^{X},\tau)$
be the product space of two multi-fuzzy topological spaces $(X_{1},\overset{n}{\underset{i=1}{\prod}}I_{i}^{X_{1}},\tau_{1})$
and $(X_{2},\overset{n}{\underset{i=1}{\prod}}I_{i}^{X_{2}},\tau_{2})$ respectively,
Let $a\in X_{1}$ (or $X_{2}$). Then the mapping $f:(X_{2},\overset{n}{\underset{i=1}{\prod}}I_{i}^{X_{2}},\tau_{2})$\\ $\rightarrow(X,\overset{n}{\underset{i=1}{\prod}}I_{i}^{X},\tau)$
(or $f:(X_{1},\overset{n}{\underset{i=1}{\prod}}I_{i}^{X_{1}},\tau_{1})\rightarrow(X,\overset{n}{\underset{i=1}{\prod}}I_{i}^{X},\tau)$)
defined by $f(x_{2})=(a,x_{2})$ (or $f(x_{1})=(x_{1},a)$) is multi
fuzzy continuous $\forall x_{2}\in X_{2}$(or $\forall x_{1}\in X_{1})$. \end{Proposition}
\begin{proof} Let $\pi_{j}:(X,\overset{n}{\underset{i=1}{\prod}}I_{i}^{X},\tau)\rightarrow(X_{j},\overset{n}{\underset{i=1}{\prod}}I_{i}^{X_{j}},\tau_{j}),j=1,2$
be the projection mappings. Now $\pi_{1}\circ f:(X_{2},\overset{n}{\underset{i=1}{\prod}}I_{i}^{X_{2}},\tau_{2})\rightarrow(X_{1},\overset{n}{\underset{i=1}{\prod}}I_{i}^{X_{1}},\tau_{1})$
is such that $\pi_{1}(f(x_{2}))=a,$ $\forall x_{2}\in X_{2}$ and
$\pi_{2}\circ f:(X_{2},\overset{n}{\underset{i=1}{\prod}}I_{i}^{X_{2}},\tau_{2})\rightarrow(X_{2},\overset{n}{\underset{i=1}{\prod}}I_{i}^{X_{2}},\tau_{2})$
is such that $\pi_{2}(f(x_{2}))=x_{2},$ $\forall x_{2}\in X_{2}.$
So, by Proposition \ref{mftvs3.7} and  Proposition \ref{mftvs3.8}, mappings $\pi_{2}\circ f$
and $\pi_{1}\circ f$ are multi-fuzzy continuous. Therefore by Proposition \ref{mftvs3.6},
$f$ is multi-fuzzy continuous. \end{proof}
\begin{Proposition} \label{mftvs3.10}Let $(X,\overset{n}{\underset{i=1}{\prod}}I_{i}^{X},\tau)$
be the product space of two multi-fuzzy topological spaces $(X_{1},\overset{n}{\underset{i=1}{\prod}}I_{i}^{X_{1}},\tau_{1})$
and $(X_{2},\overset{n}{\underset{i=1}{\prod}}I_{i}^{X_{2}},\tau_{2})$ and
$(Y,\overset{n}{\underset{i=1}{\prod}}I_{i}^{Y},\nu)$ be the product space
of two multi-fuzzy topological spaces $(Y_{1},\overset{n}{\underset{i=1}{\prod}}I_{i}^{Y_{1}},\nu_{1})$
and $(Y_{2},\overset{n}{\underset{i=1}{\prod}}I_{i}^{Y_{2}},\nu_{2})$. If the
mappings $f_{j}$ of $(X_{j},\overset{n}{\underset{i=1}{\prod}}I_{i}^{X_{j}},\tau_{j})$ into $(Y_{j},\overset{n}{\underset{i=1}{\prod}}I_{i}^{Y_{j}},\nu_{j}),$
$j=1,2$ are multi-fuzzy open, then the product mapping $f=f_{1}\times f_{2}$
from $(X,\overset{n}{\underset{i=1}{\prod}}I_{i}^{X},\tau)$ into $(Y,\overset{n}{\underset{i=1}{\prod}}I_{i}^{Y},\nu)$ defined by $f(x_{1},x_{2})=(f_{1}(x_{1}),f_{2}(x_{2}))$
is multi-fuzzy open.\end{Proposition}
\begin{proof} Let $U\in\tau.$ Then there exist multi-fuzzy open sets
$U_{jm}\in\tau_{j},j=1,2,m\in\triangle$ such that $U=\underset{m\in\triangle}{\sqcup}[U_{1m}\times U_{2m}].$
\\
Now $f(U)=\underset{m\in\triangle}{\sqcup}[f_{1}(U_{1m})\times f_{2}(U_{2m})]$.
\\
Since $f_{j},j=1,2$ are multi-fuzzy open, $f(U)$ is multi-fuzzy
open in $\nu$ and hence the product mapping $f=f_{1}\times f_{2}$
of $(X,\overset{n}{\underset{i=1}{\prod}}I_{i}^{X},\tau)$ into $(Y,\overset{n}{\underset{i=1}{\prod}}I_{i}^{Y},\nu),$
defined by $f(x_{1},x_{2})=(f_{1}(x_{1}),f_{2}(x_{2}))$ is multi-fuzzy
open. \end{proof}
\begin{Proposition}\label{mftvs3.11} Let $(X,\overset{n}{\underset{i=1}{\prod}}I_{i}^{X},\tau)$
be the product space of two multi-fuzzy topological spaces $(X_{1},\overset{n}{\underset{i=1}{\prod}}I_{i}^{X_{1}},\tau_{1})$
and $(X_{2},\overset{n}{\underset{i=1}{\prod}}I_{i}^{X_{2}},\tau_{2})$ and
$(Y,\overset{n}{\underset{i=1}{\prod}}I_{i}^{Y},\nu)$ be the product space
of two multi-fuzzy topological spaces $(Y_{1},\overset{n}{\underset{i=1}{\prod}}I_{i}^{Y_{1}},\nu_{1})$
and $(Y_{2},\overset{n}{\underset{i=1}{\prod}}I_{i}^{Y_{2}},\nu_{2})$. If the
mappings $f_{j}$ of $(X_{j},\overset{n}{\underset{i=1}{\prod}}I_{i}^{X_{j}},\tau_{j})$ into $(Y_{j},\overset{n}{\underset{i=1}{\prod}}I_{i}^{Y_{j}},\nu_{j}),$
$j=1,2$ are multi-fuzzy continuous, then the product mapping $f=f_{1}\times f_{2}$
from $(X,\overset{n}{\underset{i=1}{\prod}}I_{i}^{X},\tau)$ into $(Y,\overset{n}{\underset{i=1}{\prod}}I_{i}^{Y},\nu)$ defined by $f(x_{1},x_{2})=(f_{1}(x_{1}),f_{2}(x_{2}))$
is multi-fuzzy continuous.\end{Proposition}
\begin{proof} Since $\left(\pi_{Y_{1}}\circ f\right)(x_{1},x_{2})=\pi_{Y_{1}}(f_{1}(x_{1}),f_{2}(x_{2}))=f_{1}(x_{1})=f_{1}[\pi_{X_{1}}(x_{1},x_{2})]=\left(f_{1}\circ\pi_{X_{1}}\right)(x_{1},x_{2}),$
$\forall(x_{1},x_{2})\in X_{1}\times X_{2},$ $\pi_{Y_{1}}\circ f=f_{1}\circ\pi_{X_{1}}.$
Also, $f_{1}$ and $\pi_{X_{1}}$ are multi-fuzzy continuous and hence
from Proposition \ref{mftvs2.4}, $\pi_{Y_{1}}\circ f$ is multi-fuzzy continuous.\\
Similarly, $\pi_{Y_{2}}\circ f$ is multi-fuzzy continuous. Therefore
from Proposition \ref{mftvs3.6}, $f$ is multi-fuzzy continuous. \end{proof}
\begin{Definition}\label{mftvs3.12} A multi-fuzzy topological space $(X,\overset{n}{\underset{i=1}{\prod}}I_{i}^{X},\tau)$
is said to be second countable if there exists a countable open base
$\mathcal{B}$ for $\tau.$\end{Definition}
\begin{Proposition}\label{mftvs3.13} Let $(X_{1},\overset{n}{\underset{i=1}{\prod}}I_{i}^{X_{1}},\tau_{1})$
and $(X_{2},\overset{n}{\underset{i=1}{\prod}}I_{i}^{X_{2}},\tau_{2})$ be two
second countable multi-fuzzy topological spaces. Then their product
space $(X,\overset{n}{\underset{i=1}{\prod}}I_{i}^{X},\tau)$ is also second
countable. \end{Proposition}
\begin{proof} Let $\mathcal{B}_{1}=\{B_{j}:j\in J\}$ and $\mathcal{B}_{2}=\{B_{k}^{\prime}:k\in K\}$
be countable open base for $\tau_{1}$ and $\tau_{2}$ respectively.
Now, $\mathcal{B}=\{F\times G:F\in\tau_{1},G\in\tau_{2}\}$ is an
open base for $\tau.$ \\
Any $F\in\tau_{1},G\in\tau_{2}$ can be written as $F=\underset{j\in J_{F}}{\sqcup}B_{j}$
and $G=\underset{k\in K_{G}}{\sqcup}B_{k}^{\prime}$, for some $J_{F}\subseteq J$
and $K_{G}\subseteq K$. \\
Then $F\times G=\left(\underset{j\in J_{F}}{\sqcup}B_{j}\right)\sqcap\left(\underset{k\in K_{G}}{\sqcup}B_{k}^{\prime}\right)$\\
$=\underset{(j,k)\in J_{F}\times K_{G}}{\sqcup}(B_{j}\sqcap B_{k}^{\prime})$\\
$=\underset{(j,k)\in J_{F}\times K_{G}}{\sqcup}(B_{j}\times B_{k}^{\prime}).$\\
Since $J,K$ are countable, hence $\mathcal{B}$ is countable.\end{proof}
\begin{Definition} \label{mftvs3.14} Let $(X,\overset{n}{\underset{i=1}{\prod}}I_{i}^{X},\tau)$
be a multi-fuzzy topological space. A family $\mathcal{A}$ of multi-fuzzy
sets is a cover of a multi-fuzzy set $F$ if $F\sqsubseteq\sqcup\{A:A\in\mathcal{A}\}.$
It is an  open cover if each member of $\mathcal{A}$ is a multi-fuzzy open 
set. A subcover of $\mathcal{A}$ is a subfamily which is also a cover.\end{Definition}
\begin{Definition}\label{mftvs3.15} A multi-fuzzy topological space $(X,\overset{n}{\underset{i=1}{\prod}}I_{i}^{X},\tau)$
is said to be compact if each open cover of the space has a finite
sub-cover.\end{Definition}
\begin{Definition}\label{mftvs3.16} Let $X$ be a non-empty set and $Q$
be a subset of $X.$ A family $\mathcal{A}$ of multi-fuzzy sets is a cover
of $Q$ if $\underset{A\in\mathcal{A}}{sup}\mu_{A_{i}}(x)=1,\forall x\in X,\forall i=1,2,..,n.$\end{Definition}
\begin{Proposition} \label{mftvs3.17}Let $(X_{1},\overset{n}{\underset{i=1}{\prod}}I_{i}^{X_{1}},\tau_{1})$
and $(X_{2},\overset{n}{\underset{i=1}{\prod}}I_{i}^{X_{2}},\tau_{2})$ be two
compact multi-fuzzy topological spaces. Then their product space $(X,\overset{n}{\underset{i=1}{\prod}}I_{i}^{X},\tau)$
is also compact. \end{Proposition}

\end{document}